\documentclass[12pt]{article}
\usepackage{graphicx}
\usepackage{amsmath,amsthm,amssymb,amsfonts,euscript,enumerate}
\usepackage{amsthm}
\usepackage{color,wrapfig}
\usepackage{pxfonts}
\usepackage{verbatim}
\newtheorem{thm}{Theorem}[section]
\newtheorem{cor}[thm]{Corollary}
\newtheorem{lem}[thm]{Lemma}

\newtheorem{rmk}[thm]{Remark}

\newcommand{\R}{{\mathbb{R}}}

\newcommand{\3}{\varepsilon}
\newcommand{\4}{\widetilde}

\def\ni{\noindent}

\begin{document}
\title{Some properties of the Yamabe soliton and the related nonlinear elliptic 
equation} 
\author{Shu-Yu Hsu\\
Department of Mathematics\\
National Chung Cheng University\\
168 University Road, Min-Hsiung\\
Chia-Yi 621, Taiwan, R.O.C.\\
e-mail: syhsu@math.ccu.edu.tw}
\date{Nov 3, 2012}
\smallbreak \maketitle
\begin{abstract}
We will prove the non-existence of positive radially symmetric solution of the 
nonlinear elliptic equation $\frac{n-1}{m}\Delta v^m+\alpha v+\beta x\cdot
\nabla u=0$ in $\R^n$ when $n\ge 3$, $0<m\le\frac{n-2}{n}$, $\alpha<0$ and 
$\beta\le 0$. Let $n\ge 3$ and $g=v^{\frac{4}{n+2}}dx^2$ be a metric on $\R^n$ 
where $v$ is a radially symmetric solution of the above elliptic 
equation in $\R^n$ with $m=\frac{n-2}{n+2}$, $\alpha=\frac{2\beta+\rho}{1-m}$ 
and $\rho\in\R$. For $n\ge 3$, $m=\frac{n-2}{n+2}$, we will prove that 
$\lim_{r\to\infty}r^2v^{1-m}(r)=\frac{(n-1)(n-2)}{\rho}$ if 
$\beta>\frac{\rho}{n-2}>0$, 
the scalar curvature $R(r)\to\rho$ as $r\to\infty$ if either 
$\beta>\frac{\rho}{n-2}>0$ or $\rho=0$ and $\alpha>0$ holds, and 
$\lim_{r\to\infty}R(r)=0$ if $\rho<0$ and $\alpha>0$. We give a simple 
different proof of a result of P.~Daskalopoulos and N.~Sesum \cite{DS2} 
on the positivity of the sectional curvature of rotational symmetric 
Yamabe solitons  $g=v^{\frac{4}{n+2}}dx^2$ with $v$ satisfying the above equation
with $m=\frac{n-2}{n+2}$. We will also find the exact value of the sectional 
curvature of such Yamabe solitons at the origin and at infinity. 
\end{abstract}

\vskip 0.2truein



\vskip 0.2truein
\setcounter{section}{0}

\section{Introduction}
\setcounter{equation}{0}
\setcounter{thm}{0}

In this paper we will study various properties of the solutions of the following
nonlinear degenerate elliptic equation,
\begin{equation}\label{elliptic-eqn}
\frac{n-1}{m}\Delta v^m+\alpha v+\beta x\cdot\nabla v=0,\quad v>0,\quad
\mbox{ in }\R^n
\end{equation}
where
\begin{equation}\label{m-range}
0<m\le\frac{n-2}{n}\quad\mbox{ and }\quad n\ge 3.
\end{equation}
It is proved in \cite{DS2} by P.~Daskalopoulos and N.~Sesum that if a metric 
$g$ is a complete locally conformally flat Yamabe gradient soliton with 
positive sectional curvature, then $g=v^{\frac{4}{n+2}}dx^2$ 
where $v$ is a radially symmetric solution of \eqref{elliptic-eqn} in 
$\R^n$ with $m=\frac{n-2}{n+2}$ for some $\beta\ge 0$ and $\alpha
=\frac{2\beta+\rho}{1-m}$ where $\rho>0$, $\rho=0$
or $\rho<0$ depending on whether $g$ is a Yamabe shrinking, steady or expanding
soliton. A similar result on the rotational symmetry of complete, noncompact,
gradient Yamabe solitons with positive Ricci curvature without assuming 
the local conformally flatness of the solitons is also proved recently
by G.~Catino, C.~Mantegazza, and L.~Mazzieri \cite{CMM}. It is also 
proved in \cite{DS2} that if $v$ is a radially symmetric 
solution of \eqref{elliptic-eqn} in $\R^n$ with $m=\frac{n-2}{n+2}$ for some 
$\beta\ge 0$ and $\alpha=\frac{2\beta+\rho}{1-m}>0$, then the metric 
$g=v^{\frac{4}{n+2}}dx^2$ is a Yamabe gradient shrinking, steady or expanding 
soliton on $\R^n$ depending on whether $\rho>0$, $\rho=0$ or $\rho<0$.

On the other hand suppose $v$ is a solution of \eqref{elliptic-eqn} with 
$0<m<1$ 
and $n\ge 3$. Then as observed by B.H.~Gilding and L.A.~Peletier \cite{GP}, 
P.~Daskalopoulos and N.~Sesum \cite{DS1}, \cite{DS2}, M.~del Pino and 
M.~S\'aez \cite{PS}, J.L.~Vazquez \cite{V1}, \cite{V2}, and others, 
the function
\begin{equation*}
u_1(x,t)=t^{-\alpha}v(xt^{-\beta})
\end{equation*}
is a solution of 
\begin{equation}\label{diffusion-eqn}
u_t=\frac{n-1}{m}\Delta u^m\quad\mbox{ in }\R^n\times (0,T)
\end{equation}
if 
\begin{equation*}\label{forward}
\alpha=\frac{2\beta-1}{1-m}
\end{equation*}
and for any $T>0$ the function 
\begin{equation*}
u_2(x,t)=(T-t)^{\alpha}v(x(T-t)^{\beta})
\end{equation*}
is a solution of \eqref{diffusion-eqn} in $\R^n\times (0,T)$ if 
\begin{equation*}\label{backward}
\alpha=\frac{2\beta+1}{1-m}>0
\end{equation*}
and the function 
\begin{equation*}
u_3(x,t)=e^{-\alpha t}v(xe^{-\beta t})
\end{equation*}
is an eternal solution of \eqref{diffusion-eqn} in 
$\R^n\times (-\infty,\infty)$ if 
\begin{equation*}\label{eternal}
\alpha=\frac{2\beta}{1-m}.
\end{equation*}
The equation \eqref{elliptic-eqn} also appears in the study of the 
extinction behaviour of the solution of \eqref{diffusion-eqn} near 
the extinction time \cite{DS1}. Hence in order to understand the 
solutions of \eqref{diffusion-eqn} and the local conformally flat
Yamabe solitons, it is important to study the various properties of 
the solutions of \eqref{elliptic-eqn}. Note that $v$ is a radially symmetric 
solution of \eqref{elliptic-eqn} if and only if
\begin{equation}\label{ode}
\frac{n-1}{m}\left((v^m)''+\frac{n-1}{r}(v^m)'\right)+\alpha v+\beta rv'=0,
\quad v>0,
\end{equation}
in $(0,\infty)$ with
\begin{equation}\label{initial-cond}
v'(0)=0,\quad v(0)=\eta
\end{equation}
for some constant $\eta>0$. In \cite{H} I have proved that for any $n$, $m$, 
satisfying \eqref{m-range} and 
$$
\alpha\le\frac{\beta (n-2)}{m}\quad\mbox{ and }\quad\beta>0,
$$
there exists a unique global solution of \eqref{ode}, \eqref{initial-cond}. 
In this paper we will prove the following non-existence results for radially 
symmetric solutions of \eqref{elliptic-eqn}.

\begin{thm}\label{nonexistence-thm}
Let $m$, $n$, satisfy \eqref{m-range} and let $\eta>0$, $\alpha<0$ and 
$\beta\le 0$. Then \eqref{ode}, \eqref{initial-cond}, has no solution in 
$(0,\infty)$.
\end{thm}

We will also prove the following property of Yamabe solitons.

\begin{thm}\label{u-R-limit-thm}
Let $n\ge 3$, $m=\frac{n-2}{n+2}$, $\alpha=\frac{2\beta+\rho}{1-m}$, and 
let $v$ be a radially solution of \eqref{elliptic-eqn}. Let $R(x)$ be the 
scalar curvature of the metric $g=v^{\frac{4}{n+2}}dx^2$ on $\R^n$. Then the 
following holds.
\begin{enumerate}
\item[(i)] If either 
\begin{equation}\label{rho-beta-cond}
\left\{\begin{aligned}
&\rho=0\\
&\beta>0\end{aligned}\right.
\qquad\mbox{ or }\qquad
\beta>\rho/(n-2)>0
\end{equation}
holds, then
\begin{equation}\label{R-limit}
\lim_{|x|\to\infty}R(x)=\rho.
\end{equation}
\item[(ii)] If $\rho<0$ and $\alpha>0$, then 
\begin{equation}\label{R-limit2}
\lim_{|x|\to\infty}R(x)=0.
\end{equation}
\item[(iii)] If $\beta>\frac{\rho}{n-2}>0$, then
\begin{equation}\label{v-limit}
\lim_{|x|\to\infty}|x|^2v^{1-m}(x)=\frac{(n-1)(n-2)}{\rho}
\end{equation} 
and
\begin{equation}\label{r-v-upper-lower-bd}
0<|x|^2v^{1-m}(x)<\frac{(n-1)(n-2)}{\rho}\quad\forall x\ne 0.
\end{equation} 
\end{enumerate}
\end{thm}

We will give a simple different proof of the result of 
P.~Daskalopoulos and N.~Sesum \cite{DS2} on the positivity of the 
sectional curvature of Yamabe solitons. We also find the exact value 
of the sectional curvature of such Yamabe solitons at the origin and at 
infinity. More precisely we will prove the following theorem.

\begin{thm}\label{positivity-thm}
Let $g=v^{\frac{4}{n+2}}dx^2$ be a rotationally symmetric metric 
such that $v$ satisfies \eqref{elliptic-eqn} with $m=\frac{n-2}{n+2}$, 
$\alpha=\frac{2\beta+\rho}{1-m}$, which is either a steady Yamabe soliton
($\rho=0$) with $\alpha>0$, or a Yamabe expanding soliton ($\rho<0$) with
$\alpha>0$, or a Yamabe shrinking soliton ($\rho>0$) with $\beta>\rho/(n-2)$. 
Then $g$ has strictly positive sectional curvature. If $K_0$ and $K_1$ 
are the sectional curvatures of the $2$-planes perpendicular to and 
tangent to the spheres $\{x\}\times S^{n-1}$ respectively, then
\begin{equation}\label{K(0)}
K_0(0)=K_1(0)=\frac{2\beta+\rho}{n(n-1)},
\end{equation}
\begin{equation}\label{K0-limit}
\lim_{r\to\infty}K_0(r)=0,
\end{equation}
and
\begin{equation}\label{K1-limit}
\lim_{r\to\infty}K_1(r)=\left\{\begin{aligned}
&\frac{\rho}{(n-1)(n-2)}\quad\mbox{ if $g$ is a shrinking Yamabe soliton}\\
&0\qquad\qquad\qquad\mbox{ if $g$ is a steady or expanding Yamabe soliton}.
\end{aligned}\right.
\end{equation}
\end{thm}

Note that the results \eqref{K(0)}, \eqref{K0-limit} and \eqref{K1-limit}, 
are entirely new. 
The plan of the paper is as follows. We will prove the non-existence of 
solutions of \eqref{ode}, \eqref{initial-cond}, and prove various 
properties of solutions of \eqref{ode}, \eqref{initial-cond}, in section 
two. In section three we will prove the asymptotic behaviour of the 
metric and the scalar curvature of the Yamabe solitons. In section 
four we will give another proof of the positivity of the sectional 
curvature of Yamabe solitons. We will also prove the asymptotic behaviour of 
the sectional curvature of the Yamabe solitons as $|x|\to\infty$. 

\section{Non-existence and properties of solutions of the nonlinear elliptic 
equation}
\setcounter{equation}{0}
\setcounter{thm}{0}

For any $\beta\in\R$, $\alpha\ne 0$, let $k=\beta/\alpha$.
We first recall a result of \cite{H}.

\begin{lem}\label{v'-bd-lem}(Lemma 2.1 of \cite{H})
Let $m$, $\alpha\ne 0$, $\beta\ne 0$, satisfy \eqref{m-range} and 
\begin{equation}\label{alpha-beta-relation}
\frac{m\alpha}{\beta}\le n-2.
\end{equation}
For any $R_0>0$ and $\eta>0$, let $v$ be the solution of \eqref{ode}, 
\eqref{initial-cond}, in $(0,R_0)$. Then  
\begin{equation}\label{basic-monotone-ineqn}
v+krv'(r)>0\quad\mbox{ in }[0,R_0)
\end{equation}
and
\begin{equation}\label{v'-bd0}
\left\{\begin{aligned}
&v'(r)<0\quad\mbox{ in }(0,R_0)\quad\mbox{ if }\alpha>0\\
&v'(r)>0\quad\mbox{ in }(0,R_0)\quad\mbox{ if }\alpha<0.\end{aligned}\right.
\end{equation}
\end{lem}

\noindent{\ni{\it Proof of Theorem \ref{nonexistence-thm}}:}
Suppose there exists a solution $v$ for \eqref{ode}, \eqref{initial-cond}, 
in $(0,\infty)$. Multiplying \eqref{ode} by $r^{n-1}$ and integrating we have
\begin{align}\label{v-integral-eqn}
\frac{n-1}{m}r^{n-1}(v^m)'(r)=&-\alpha\int_0^rz^{n-1}v(z)\,dz
-\beta\int_0^rz^nv'(z)\,dz\notag\\
=&-\beta r^nv(r)+(n\beta -\alpha)\int_0^rz^{n-1}v(z)\,dz\quad\forall r>0.
\end{align}
We now divide the proof into three cases.

\noindent{\bf Case 1}: $0>n\beta>\alpha$.

\noindent By \eqref{v-integral-eqn},
\begin{equation*}
\frac{n-1}{m}r^{n-1}(v^m)'(r)\ge|\beta|r^nv(r)\quad\Rightarrow\quad
(n-1)v^{m-2}v'(r)\ge|\beta|r\quad\forall r>0.
\end{equation*}

\noindent{\bf Case 2}: $0>\alpha\ge n\beta$. 

\noindent By \eqref{m-range}, \eqref{alpha-beta-relation} holds.
Since $\alpha<0$, by Lemma \eqref{v'-bd-lem} $v'(r)>0$ for any $r>0$.
Then by \eqref{v-integral-eqn},
\begin{equation*}
\frac{n-1}{m}r^{n-1}(v^m)'(r)\ge-\beta r^nv(r)-\frac{\alpha-n\beta}{n}r^nv(r)
=\frac{|\alpha|}{n}r^nv(r)
\end{equation*}
for any $r>0$. Hence
\begin{equation*}
(n-1)v^{m-2}v'(r)\ge\frac{|\alpha|}{n}r\quad\forall r>0.
\end{equation*}
By case 1 and case 2,
\begin{equation}\label{v-integral-ineqn}
v^{m-2}v'(r)\ge C_1r\quad\forall r>0
\end{equation}
where
$$
C_1=\frac{1}{n-1}\min\left(\frac{|\alpha|}{n},|\beta|\right).
$$
Integrating \eqref{v-integral-ineqn} over $(0,r)$,
\begin{align*}
&\frac{1}{1-m}(\eta^{m-1}-v(r)^{m-1})\ge\frac{C_1}{2}r^2\quad\forall r>0\\
\Rightarrow\quad&v^{1-m}(r)\ge\left(\eta^{m-1}-\frac{(1-m)C_1}{2}r^2\right)^{-1}
\to\infty\quad\mbox{ as }\quad r\nearrow\sqrt{2/(C_1(1-m))}\eta^{\frac{m-1}{2}}.
\end{align*}
Contradiction arises. Hence \eqref{ode}, \eqref{initial-cond}, has no
global solution when either case 1 or case 2 holds.

\noindent{\bf Case 3}: $\alpha<0$ and $\beta=0$.  

\noindent By \eqref{v-integral-eqn},
\begin{align}\label{v'-upper-bd}
&\frac{n-1}{m}r^{n-1}(v^m)'(r)=|\alpha|\int_0^rs^{n-1}v(s)\,ds>0
\quad\forall r>0\notag\\
\Rightarrow\quad&0<\frac{n-1}{m}r^{n-1}(v^m)'(r)\le\frac{|\alpha|}{n}r^nv(r)
\quad\forall r>0\notag\\
\Rightarrow\quad&0<\frac{n-1}{r}(v^m)'(r)\le\frac{m|\alpha|}{n}v(r)
\quad\forall r>0.
\end{align}
Hence by \eqref{ode} and \eqref{v'-upper-bd},
\begin{align}
|\alpha|v=&\frac{n-1}{m}\left((v^m)''(r)+\frac{n-1}{r}(v^m)'(r)\right)\notag\\
\le&\frac{n-1}{m}\left((v^m)''(r)+\frac{m|\alpha|}{n}v(r)\right)
\qquad\qquad\qquad\quad\,\forall r>0\notag\\
\Rightarrow\qquad\qquad&(v^m)''(r)\ge\frac{m|\alpha|}{n(n-1)}v(r)
\qquad\qquad\qquad\qquad\quad\forall r>0\label{vm''-lower-bd}\\
\Rightarrow\qquad\qquad&(v^m)_r^2(r)\ge\frac{2m^2|\alpha|}{n(n-1)(1+m)}
(v^{1+m}(r)-\eta^{1+m})\quad\forall r>0\label{vm'-lower-bd}.
\end{align}
Since $v(r)>v(0)$ for all $r>0$, by \eqref{vm''-lower-bd} $v(r)\to\infty$
as $r\to\infty$. Hence there exists a constant $R_1>0$ such that
\begin{equation}\label{v^{m+1}-lower-bd}
v^{1+m}(r)>2\eta^{1+m}\quad\forall r\ge R_1.
\end{equation}
By \eqref{vm'-lower-bd} and \eqref{v^{m+1}-lower-bd}, there exists a constant
$C_2>0$ such that
\begin{align*}
&(v^m)_r^2(r)\ge C_2^2m^2v^{1+m}(r)\quad\forall r\ge R_1\\
\Rightarrow\quad&v^{-\frac{3-m}{2}}v'(r)\ge C_2\qquad\qquad\forall r\ge R_1\\
\Rightarrow\quad&v^{\frac{1-m}{2}}(r)\ge\left(v(R_1)^{\frac{m-1}{2}}
-\frac{(1-m)C_2}{2}(r-R_1)\right)^{-1}\to\infty\quad\mbox{ as }r\nearrow
R_1+\frac{2v(R_1)^{\frac{m-1}{2}}}{(1-m)C_2}.
\end{align*}
Contradiction arises. Hence \eqref{ode}, \eqref{initial-cond}, has no
global solution in case 3 and the theorem follows.

\hfill$\square$\vspace{6pt}

Note that if $\alpha=\beta=0$, then the constant fucntion $v=\eta$ is a solution
of \eqref{ode}, \eqref{initial-cond}. Hence Theorem \ref{nonexistence-thm} is sharp.
As a result of Theorem \ref{nonexistence-thm} we get the following result of 
\cite{DS2}.

\begin{cor}(cf. Claim 3.1 of \cite{DS2})
Let $n\ge 3$ and $\eta>0$. Suppose $g=v^{\frac{4}{n+2}}dx^2$ is a complete 
locally flat gradient Yamabe soliton with radially symmetric $v$ 
satisfying \eqref{ode}, \eqref{initial-cond}, with $m=\frac{n-2}{n+2}$, 
$\alpha=\frac{2\beta+\rho}{1-m}$, where $\rho=0$ for Yamabe 
steady soliton and $\rho<0$ for Yamabe expanding soliton respectively. 
Then $\beta>0$ if $\rho<0$ and $\beta\ge 0$ if $\rho=0$.
\end{cor}

\begin{lem}\label{r2v(1-m)+v'-bd-lem}
Let $n\ge 3$, $0<m<1$, $\alpha\ge n\beta$ and $\alpha>0$, and $\eta>0$. 
Suppose $v$ is a solution of \eqref{ode}, \eqref{initial-cond}. Then 
\begin{equation}\label{r^2v^{1-m}-bd}
0<r^2v^{1-m}(r)\le\frac{2n(n-1)}{\alpha(1-m)}\quad\forall r>0
\end{equation}
and
\begin{equation}\label{v-v'-bd1}
\frac{\alpha}{n(n-1)}r^2v^{1-m}(r)+\frac{rv'(r)}{v(r)}\le 0\quad\forall r>0.
\end{equation}
Hence $v'(r)<0$ for all $r>0$.
\end{lem}
\begin{proof}
By \eqref{v-integral-eqn}, $(v^m)'(r)<0$ for all $r>0$. Hence $v'(r)<0$ for all
$r>0$. Then by \eqref{v-integral-eqn},
\begin{align}\label{v'-negative}
&\frac{n-1}{m}r^{n-1}(v^m)'(r)\le-\beta r^nv(r)-(\alpha-n\beta)
\int_0^rz^{n-1}v(r)\,dz
=-\frac{\alpha}{n}r^nv(r)\quad\forall r>0\notag\\
\Rightarrow\quad&v^{m-2}(r)v'(r)\le-\frac{\alpha}{n(n-1)}r
\qquad\qquad\quad\forall r>0
\end{align}
and \eqref{v-v'-bd1} follows. Integrating \eqref{v'-negative} over $(0,r)$ 
and simplifying,
\begin{equation*}
v(r)\le\left(\eta^{m-1}+\frac{\alpha(1-m)}{2n(n-1)}r^2\right)^{-\frac{1}{1-m}}
\le\left(\frac{2n(n-1)}{\alpha(1-m)}r^{-2}\right)^{\frac{1}{1-m}}\quad\forall r>0
\end{equation*}
and \eqref{r^2v^{1-m}-bd} follows.
\end{proof}

\begin{rmk}\label{alpha-beta-range-rmk}
If $n\ge 3$, $0<m\le \frac{n-2}{n}$ and $m\alpha\ge\beta (n-2)$, then
$\alpha\ge n\beta$.
\end{rmk}

\section{Asymptotic behaviour of Yamabe solitons}
\setcounter{equation}{0}
\setcounter{thm}{0}

We will assume from now on that $n\ge 3$, $m=\frac{n-2}{n+2}$, and 
$k=\beta/\alpha$ for any $\alpha\ne 0$ and $\beta\in\R$ for the rest 
of the paper. In this section we will study the asymptotic 
behaviour of locally conformally flat Yamabe solitons with positive
sectional curvature. By the results of \cite{DS2} we can write the metric of
such soliton as $g=v^{\frac{4}{n+2}}dx^2$ where $v$ is a radially
symmetric solution of \eqref{elliptic-eqn} in $\R^n$ with $m=\frac{n-2}{n+2}$ 
for some $\beta\ge 0$ and $\alpha=\frac{2\beta+\rho}{1-m}$ where $\rho>0$, 
$\rho=0$ or $\rho<0$, depending on whether $g$ is a Yamabe shrinking, 
steady or expanding soliton. We will use $R$ to denote the scalar curvature 
of the metric $g$ and let
\begin{equation*}
w(r)=r^2v^{1-m}(r),\quad s=\log r,\quad\4{w}(s)=w(r). 
\end{equation*}
\begin{lem}\label{soliton-bd-lem}
Let $g=v^{\frac{4}{n+2}}dx^2$ be a rotational symmetric Yamabe soliton with 
$v$ satisfying \eqref{elliptic-eqn} with $m=\frac{n-2}{n+2}$ and 
$\alpha=\frac{2\beta+\rho}{1-m}$ for some constant $\rho\in\R$. Then 
the following holds.
\begin{enumerate}
\item[(i)] If $g$ is a Yamabe steady soliton or a Yamabe expanding soliton, 
then $R>0$ if $\alpha>0$.
\item[(ii)] If $g$ is a Yamabe expanding soliton, then $R<0$ if $\alpha
<0<\beta$.
\item[(iii)] If $g$ is a Yamabe shrinking soliton with $\beta>\rho/(n-2)$, 
then $R>\rho$. 
\item[(iv)] If $g$ is a Yamabe shrinking, steady, or expanding soliton 
with $\alpha>0$, then 
\begin{equation}\label{R-bd}
0\le R\le\alpha (1-m).
\end{equation}
\item[(v)] Suppose $g$ is either a Yamabe shrinking soliton with 
$\beta>\rho/(n-2)$ or a Yamabe steady or expanding soliton with 
$\alpha>0$. Then 
\begin{equation}\label{v'-v-lower-bd}
1+\frac{1-m}{2}\frac{rv'(r)}{v(r)}>0\quad\forall r\ge 0
\end{equation}
and $w'(r)>0$ for any $r>0$.
\end{enumerate} 
\end{lem}
\begin{proof}
(i), (ii) and (iii) is proved in Proposition 4.1 of \cite{DS2}. For the sake of 
completeness we will give a simple different proof of (i) and (ii) here.
As observed in \cite{DS2} since the scalar curvature satisfies (P.184 of
\cite{SY}),
\begin{equation}
R=-\frac{4(n-1)}{n-2}\cdot\frac{\Delta v^m}{v},
\end{equation}
by \eqref{elliptic-eqn},
\begin{equation}\label{R-eqn}
R(r)=(1-m)\left(\alpha +\beta\frac{rv'(r)}{v(r)}\right)
=\alpha(1-m)\left(1+k\frac{rv'(r)}{v(r)}\right).
\end{equation}
Note that for Yamabe steady soliton and Yamabe expanding soliton we have
$\alpha=\frac{2\beta+\rho}{1-m}$ with $\rho=0$ and $\rho<0$ respectively. 
Hence \eqref{alpha-beta-relation} holds for Yamabe steady soliton with 
$\alpha>0$ and for Yamabe expanding soliton with either $\alpha>0$ or 
$\alpha<0<\beta$. Thus by Lemma \ref{v'-bd-lem}, \eqref{basic-monotone-ineqn} 
holds for all $r>0$ in both cases (i) and (ii). By \eqref{basic-monotone-ineqn} 
and \eqref{R-eqn}, (i) and (ii) of the lemma follows.

To prove (iv) we first observe that by the result of \cite{DS2} if 
$g$ is a Yamabe shrinking soliton, then 
$R\ge 0$. This together with (i) imply the first inequality $R\ge 0$
of \eqref{R-bd}. We next observe that under the hypothesis of (iv)
by Lemma \ref{v'-bd-lem}, Lemma \ref{r2v(1-m)+v'-bd-lem} and 
Remark \ref{alpha-beta-range-rmk}, we have $v'(r)<0$ for all $r>0$. Hence
by \eqref{R-eqn} we get \eqref{R-bd}.

Suppose $g$ is either a Yamabe shrinking soliton with 
$\beta>\rho/(n-2)$ or a Yamabe steady or expanding soliton with 
$\alpha>0$. Then by (i), (iii), and \eqref{R-eqn},
\begin{equation*}
\rho+2\beta\left(1+\frac{1-m}{2}\frac{rv'(r)}{v(r)}\right)>\rho
\quad\forall r\ge 0
\end{equation*}
and \eqref{v'-v-lower-bd} follows. Hence
\begin{equation}\label{w'-eqn}
w'(r)=2rv^{1-m}\left(1+\frac{1-m}{2}\frac{rv'(r)}{v(r)}\right)>0\quad\forall
r>0
\end{equation}
and (v) follows.
\end{proof}

\begin{lem}\label{v-infty-growth2}
Let $n\ge 3$, $m=\frac{n-2}{n+2}$, $\eta>0$, $\beta>\frac{\rho}{n-2}>0$, 
$\alpha=\frac{2\beta+\rho}{1-m}$, be such that $n\beta>\alpha$. 
Suppose $v$ is a solution of \eqref{ode}, \eqref{initial-cond}. Then 
\begin{equation}\label{v-decay}
0<r^2v^{1-m}(r)\le\frac{(n-1)(n-2)}{\rho}\quad\forall r>0
\end{equation}
and for any $0<\delta<\frac{\rho}{n(1-m)-2}$ there exists a constant
$R_1>1$ such that 
\begin{equation}\label{vm-v'/v-bd2}
\frac{1}{n-1}\left[\frac{\rho}{n(1-m)-2}-\delta\right]r^2v^{1-m}(r)
+\frac{rv'(r)}{v(r)}\le 0\quad\forall r\ge R_1.
\end{equation}
\end{lem}
\begin{proof}
We first claim that
\begin{equation}\label{term-ratio}
\limsup_{r\to\infty}\frac{\int_0^rz^{n-1}v(z)\,dz}{r^nv(r)}\le\frac{1-m}{n(1-m)-2}.
\end{equation}
By (v) of Lemma \ref{soliton-bd-lem} there exists a constant $C>0$ 
such that $w(r)\ge C$ for any $r>1$. Hence
\begin{align}
&r^nv(r)=r^{n-\frac{2}{1-m}}w^{\frac{1}{1-m}}(r)\ge Cr^{n-\frac{2}{1-m}}\quad\forall
r>1\label{r^n-v-w-eqn}\\
\Rightarrow\quad&r^nv(r)\to\infty\quad\mbox{ as }r\to\infty\label{r^n-v-limit}.
\end{align}
We now divide the claim into two cases.

\noindent{\bf Case 1}: $\int_0^{\infty}z^{n-1}v(z)\,dz<\infty$.

\noindent By \eqref{r^n-v-limit} we get \eqref{term-ratio}.

\noindent{\bf Case 2}: $\int_0^{\infty}z^{n-1}v(z)\,dz=\infty$.

\noindent Since by \eqref{r^n-v-w-eqn} and (v) of Lemma \ref{soliton-bd-lem},
$$
\frac{d}{dr}(r^nv(r))=\left(n-\frac{2}{1-m}\right)r^{n-1}v(r)
+\frac{1}{1-m}r^{n-\frac{2}{1-m}}w^{\frac{m}{1-m}}(r)w'(r)\ge 
\left(n-\frac{2}{1-m}\right)r^{n-1}v(r)\quad\forall r>0,
$$
by the l'Hospital rule,
\begin{align*}
\limsup_{r\to\infty}\frac{\int_0^rz^{n-1}v(z)\,dz}{r^nv(r)}
=&\limsup_{r\to\infty}\frac{r^{n-1}v(r)}{\left(n-\frac{2}{1-m}\right)r^{n-1}v(r)
+\frac{1}{1-m}r^{n-\frac{2}{1-m}}w^{\frac{m}{1-m}}(r)w'(r)}\\
\le&\left(n-\frac{2}{1-m}\right)^{-1}
\end{align*}
and \eqref{term-ratio} follows. 
Let $0<\delta<\frac{\rho}{n(1-m)-2}$ and let $\3>0$ be given by
\begin{equation}\label{epsilon-delta-ineqn}
\frac{2(n-1)}{1-m}\cdot\left[\frac{\rho}{n(1-m)-2}-\delta\right]^{-1}
=\frac{2(n-1)(n(1-m)-2)}{(1-m)\rho}+\3
=\frac{(n-1)(n-2)}{\rho}+\3.
\end{equation}
Then $\3\to 0$ as $\delta\to 0$. By the above claim there exists a 
constant $R_1>1$ such that
\begin{align}\label{term-compare}
&\frac{\int_0^rz^{n-1}v(z)\,dz}{r^nv(r)}<\frac{(1-m)}{n(1-m)-2}
+\frac{\delta}{n\beta-\alpha}
\qquad\qquad\quad\forall r\ge R_1\notag\\
\Rightarrow\quad&\int_0^rz^{n-1}v(z)\,dz\le\left(
\frac{(1-m)}{n(1-m)-2}+\frac{\delta}{n\beta-\alpha}\right)r^nv(r)
\quad\forall r\ge R_1.
\end{align}
By \eqref{v-integral-eqn} and \eqref{term-compare},
\begin{align}\label{v^m'-integral-eqn}
\frac{n-1}{m}r^{n-1}(v^m)'(r)
\le&-\beta r^nv(r)+\left(\frac{(n\beta -\alpha)(1-m)}{n(1-m)-2}+\delta\right)
r^nv(r)\notag\\
\le&-\left(\frac{\rho}{n(1-m)-2}-\delta\right)r^nv(r)
\quad\forall r\ge R_1\notag\\
\Rightarrow\quad (n-1)v^{m-2}v'(r)
\le&-\left(\frac{\rho}{n(1-m)-2}-\delta\right)r\qquad\quad\forall r\ge R_1
\end{align}
and \eqref{vm-v'/v-bd2} follows. Integrating \eqref{v^m'-integral-eqn} over
$(R_1,r)$ and simplifying, by \eqref{epsilon-delta-ineqn},
\begin{align*}
v^{1-m}(r)\le&\left(v(R_1)^{m-1}+\frac{1-m}{2(n-1)}\cdot
\left(\frac{\rho}{n(1-m)-2}-\delta\right)r^2\right)^{-1}\quad\forall r\ge R_1
\notag\\
\le&\left(\frac{(n-1)(n-2)}{\rho}+\3\right)r^{-2}\quad\forall r\ge R_1\\
\Rightarrow\quad r^2v^{1-m}(r)\le&\frac{(n-1)(n-2)}{\rho}+\3\quad\forall 
r\ge R_1.
\end{align*}
Hence by (v) of Lemma \ref{soliton-bd-lem},
\begin{equation}\label{w-epsilon-bd}
w(r)=r^2v^{1-m}(r)\le\frac{(n-1)(n-2)}{\rho}+\3\quad\forall r>0.
\end{equation}
Letting $\3\to 0$ in \eqref{w-epsilon-bd} we get \eqref{v-decay} and the
lemma follows.
\end{proof}

\begin{lem}\label{r-v-limit-lem}
Let $n\ge 3$, $m=\frac{n-2}{n+2}$, $\alpha=\frac{2\beta+\rho}{1-m}$ and 
$\beta>\frac{\rho}{n-2}>0$. Suppose $v$ is a radially symmetric solution 
of \eqref{elliptic-eqn}. Then
$$
a_0=\lim_{r\to\infty}r^2v^{1-m}(r)
$$
exists and $0<a_0<\infty$. 
\end{lem}
\begin{proof}
By Lemma \ref{r2v(1-m)+v'-bd-lem} \eqref{r^2v^{1-m}-bd} holds if 
$\alpha\ge n\beta$ and by Lemma \ref{v-infty-growth2} \eqref{v-decay} holds 
if $n\beta>\alpha$. Hence by (v) of Lemma \ref{soliton-bd-lem}, 
$r^2v^{1-m}(r)$ converges to some positive number as $r\to\infty$ and the 
lemma follows.
\end{proof}

\begin{lem}\label{v'/v-limit-lem}
Let $g=v^{\frac{4}{n+2}}dx^2$ be a rotationally symmetric metric 
which satisfies \eqref{elliptic-eqn} with $m=\frac{n-2}{n+2}$, 
$\alpha=\frac{2\beta+\rho}{1-m}$, which is either a steady Yamabe soliton
($\rho=0$) with $\alpha>0$, or a Yamabe expanding soliton ($\rho<0$) with
$\alpha>0$, or a Yamabe shrinking soliton ($\rho>0$) with 
$\beta>\frac{\rho}{n-2}$. 
Then
\begin{equation}\label{v'/v-limit}
\lim_{r\to\infty}\frac{rv'(r)}{v(r)}=\left\{\begin{aligned}
&-\frac{2}{1-m}\quad\mbox{ if $g$ is a shrinking or steady Yamabe soliton}\\
&-\frac{1}{k}\qquad\,\,\mbox{ if $g$ is an expanding Yamabe soliton}.
\end{aligned}\right.
\end{equation}
\end{lem}
\begin{proof}
By  Lemma \ref{v'-bd-lem}, Lemma \ref{r2v(1-m)+v'-bd-lem},
Remark \ref{alpha-beta-range-rmk} and Lemma \ref{soliton-bd-lem},
\begin{equation}\label{v'/v-uniform-bd}
-\frac{2}{1-m}\le\frac{rv'(r)}{v(r)}\le 0\quad\forall r>0.
\end{equation}
Let $\{r_i\}_{i=1}^{\infty}$ be a sequence such that $r_i\to\infty$ as 
$i\to\infty$. Then by \eqref{v'/v-uniform-bd} $\{r_i\}_{i=1}^{\infty}$
has a subsequence which we may assume without loss of generality
to be the sequence $\{r_i\}_{i=1}^{\infty}$ itself such that
$\frac{r_iv'(r_i)}{v(r_i)}$ converges to some number $a$ as $i\to\infty$. 
Let $a_0$ be given by Lemma \ref{r-v-limit-lem}.
We now divide the proof into three cases.

\noindent{\bf Case 1}: g is a Yamabe shrinking soliton with 
$\beta>\frac{\rho}{n-2}$.

\noindent Then by the l'Hospital rule,
\begin{align*}
a_0=&\lim_{i\to\infty}r_i^2v^{1-m}(r_i)=\lim_{i\to\infty}\frac{v^{1-m}(r_i)}{r_i^{-2}}
=\lim_{i\to\infty}\frac{(1-m)v^{-m}(r_i)v'(r_i)}{-2r_i^{-3}}\\
=&-\frac{1-m}{2}\lim_{i\to\infty}r_i^2v^{1-m}(r_i)\cdot
\lim_{i\to\infty}\frac{r_iv'(r_i)}{v(r_i)}=-\frac{1-m}{2}a_0a\\
\Rightarrow\quad a=&-\frac{2}{1-m}.
\end{align*}
Since the sequence $\{r_i\}_{i=1}^{\infty}$ is arbitrary, we get 
\begin{equation}\label{v'-limit-a}
\lim_{r\to\infty}\frac{rv'(r)}{v(r)}=-\frac{2}{1-m}.
\end{equation}

\noindent{\bf Case 2}: g is a Yamabe steady soliton with $\alpha>0$.

\noindent By Theorem 1.3 of \cite{H},
$$
\lim_{r\to\infty}\frac{r^2v^{1-m}(r)}{\log r}=a_1
$$
where $a_1=2(n-1)(n-2-mn)/[\beta(1-m)]$. Then by the l'Hospital rule,
\begin{align*}
a_1=&\lim_{i\to\infty}\frac{v^{1-m}(r_i)}{r_i^{-2}\log r_i}
=\lim_{i\to\infty}\frac{(1-m)v^{-m}(r_i)v'(r_i)}{-2r_i^{-3}\log r_i+r_i^{-3}}
=-\frac{1-m}{2}\lim_{i\to\infty}\frac{r_i^2v^{1-m}(r_i)}{\log r_i}\cdot
\lim_{i\to\infty}\frac{r_iv'(r_i)}{v(r_i)}\\
=&-\frac{1-m}{2}a_1a.
\end{align*}
Hence
\begin{equation*}
a=-\frac{2}{1-m}.
\end{equation*}
Since the sequence $\{r_i\}_{i=1}^{\infty}$ is arbitrary, we get 
\eqref{v'-limit-a}.

\noindent{\bf Case 3}: g is a Yamabe expanding soliton with $\alpha>0$.

\noindent By Theorem 1.6 of \cite{H} (cf. Theorem 3.2 of \cite{V1}), 
$$
\lim_{r\to\infty}r^2v^{2k}(r)=a_2
$$
for some constant $0<a_2<\infty$. Then by the l'Hospital rule,
\begin{equation*}
a_2=\lim_{i\to\infty}\frac{v^{2k}(r_i)}{r_i^{-2}}
=\lim_{i\to\infty}\frac{2kv^{2k-1}(r_i)v'(r_i)}{-2r_i^{-3}}
=-k\lim_{i\to\infty}r_i^2v^{2k}(r_i)\cdot
\lim_{i\to\infty}\frac{r_iv'(r_i)}{v(r_i)}
=-ka_2a.
\end{equation*}
Hence
\begin{equation*}
a=-\frac{1}{k}.
\end{equation*}
Since the sequence $\{r_i\}_{i=1}^{\infty}$ is arbitrary, we get 
\begin{equation*}\label{v'-limit-b}
\lim_{r\to\infty}\frac{rv'(r)}{v(r)}=-\frac{1}{k}.
\end{equation*}
By case 1, case 2 and case 3, the lemma follows.
\end{proof}

We are now ready for the proof of Theorem \ref{u-R-limit-thm}.

\noindent {\ni{\it Proof of Theorem \ref{u-R-limit-thm}}:} 
Suppose first \eqref{rho-beta-cond} holds. Then by \eqref{R-eqn}
and Lemma \ref{v'/v-limit-lem},
\begin{equation}\label{R-eqn2}
R=\rho+2\beta\left(1+\frac{1-m}{2}\frac{rv'(r)}{v(r)}\right)\to\rho\quad
\mbox{ as }\quad r\to\infty
\end{equation}
and (i) of Theorem \ref{u-R-limit-thm} follows. If $\rho<0$ and 
$\alpha>0$, then by \eqref{R-eqn} and Lemma \ref{v'/v-limit-lem}
we get (ii) of Theorem \ref{u-R-limit-thm}. 

We next assume that $\beta>\frac{\rho}{n-2}>0$. Let $a_0$ be as in 
Lemma \ref{r-v-limit-lem}. Then by Lemma \ref{r2v(1-m)+v'-bd-lem}, 
Lemma \ref{v-infty-growth2}, \eqref{w'-eqn} and \eqref{R-eqn2},
\begin{equation}\label{w_s-limit}
\4{w}_s(s)=rw'(r)=\frac{\4{w}(s)}{\beta}(R-\rho)\to 0\quad\mbox{ as }
\quad s=\log r\to\infty.
\end{equation}
Let $\{r_i\}_{i=1}^{\infty}$ be such that $r_i\to\infty$ as $i\to\infty$
and $s_i=e^{r_i}$. As proved in \cite{DS2} and \cite{H} $\4{w}$ satisfies
\begin{equation}\label{w-tilde-eqn}
\4{w}_{ss}=\frac{1-2m}{1-m}\cdot\frac{\4{w}_s^2}{\4{w}}
-\frac{\beta}{n-1}\4{w}\4{w}_s-\frac{\rho}{n-1}\4{w}^2
+\frac{2(n-2-nm)}{1-m}\4{w}
\end{equation}
in $(-\infty,\infty)$. Hence by Lemma \ref{r-v-limit-lem}, \eqref{w_s-limit}
and \eqref{w-tilde-eqn},
$$
a_3=\lim_{s\to\infty}\4{w}_{ss}(s)
=-\frac{a_0^2\rho}{n-1}+\frac{2(n-2-nm)}{1-m}a_0
$$
exists. Suppose $a_3\ne 0$. Without loss of generality we may assume that
$a_3>0$. Then 
$$
\lim_{s\to\infty}\4{w}_s(s)=\infty
$$ 
which contradicts \eqref{w_s-limit}. Hence $a_3=0$. Thus
\begin{align*}
\frac{a_0^2\rho}{n-1}=\frac{2(n-2-nm)}{1-m}a_0\quad\Rightarrow\quad
a_0=\frac{2(n-1)(n-2-nm)}{(1-m)\rho}=\frac{(n-1)(n-2)}{\rho}
\end{align*}
and \eqref{v-limit} follows. By \eqref{v-limit} and (v) of 
Theorem \ref{soliton-bd-lem} we get \eqref{r-v-upper-lower-bd} and 
the theorem follows.

{\hfill$\square$\vspace{6pt}}

\section{Positivity and asymptotic behaviour of the sectional 
curvature of Yamabe solitons}
\setcounter{equation}{0}
\setcounter{thm}{0}

In this section we will give a simple proof on the positivity of 
the sectional curvature of rotational symmetric Yamabe solitons of the form 
$g=v^{\frac{4}{n+2}}dx^2$ in $\R^n$, $n\ge 3$, where $v$ satisfies 
\eqref{elliptic-eqn} with 
$m=\frac{n-2}{n+2}$. We also find the exact value of the sectional curvature 
of such Yamabe solitons at the origin and at infinity. We first prove the
following improvement of Corollary 4.2 of \cite{DS2}.

\begin{lem}(cf. Corollary 4.2 of \cite{DS2})\label{R'<0}
Let $g=v^{\frac{4}{n+2}}dx^2$ be a rotationally symmetric Yamabe soliton
in $\R^n$, $n\ge 3$, such that $v$ satisfies \eqref{elliptic-eqn} with 
$m=\frac{n-2}{n+2}$ and $\alpha=\frac{2\beta+\rho}{1-m}$.
Suppose $\beta>\frac{\rho}{n-2}>0$ if $g$ is a Yamabe shrinking soliton,
and $\alpha>0$ if $g$ is a Yamabe steady or expanding soliton. Then 
the scalar curvature $R(r)$ is a strictly decreasing function of $r>0$
and $R'(r)<0$ for all $r>0$.
\end{lem}
\begin{proof}
As proved in \cite{C} and \cite{DS2}, $R$ satisfies
$$
(n-1)\Delta R+\beta (x\cdot\nabla R)v^{1-m}+R(R-\rho)v^{1-m}=0
\quad\mbox{ in }\R^n
$$
where $\Delta$, $\nabla$, are the laplacian and gradient with respect to the
Euclidean metric in $\R^n$. Hence
\begin{align}\label{R-integral}
&R''(r)+\frac{n-1}{r}R'(r)+\frac{\beta}{n-1}v(r)^{1-m}
rR'(r)=-\frac{1}{n-1}v(r)^{1-m}R(r)(R(r)-\rho)\qquad\qquad\,\,\,\,
\forall r>0\notag\\
\Rightarrow\quad&\frac{d}{dr}\left(r^{n-1}
e^{\frac{\beta}{n-1}\int_0^r\tau v(\tau)^{1-m}\,d\tau}R'(r)\right)
=-\frac{r^{n-1}}{n-1}v(r)^{1-m}R(r)(R(r)-\rho)
e^{\frac{\beta}{n-1}\int_0^r\tau v(\tau)^{1-m}\,d\tau}\quad\forall r>0\notag\\
\Rightarrow\quad&R'(r)
=-\frac{\int_0^rz^{n-1}v(z)^{1-m}R(z)(R(z)-\rho)
e^{\frac{\beta}{n-1}\int_0^z\tau v(\tau)^{1-m}\,d\tau}\,dz}
{(n-1)r^{n-1}e^{\frac{\beta}{n-1}\int_0^r\tau v(\tau)^{1-m}\,d\tau}}\qquad\qquad\qquad\quad
\forall r>0.
\end{align}
By (i) and (iii) of Lemma \ref{soliton-bd-lem} the term $R(r)(R(r)-\rho)$
is always positive. Hence $R'(r)<0$ for all $r>0$ and the lemma follows.
\end{proof}

\noindent We are now ready for the proof of Theorem \ref{positivity-thm}:

\noindent{\ni {\it Proof of Theorem \ref{positivity-thm}}}:
Positivity of the sectional curvature of Yamabe soliton is proved in 
\cite{DS2}. Since the proof in \cite{DS2} is hard, we will give a simple
proof of this result in this paper. Similar to \cite{DS2}
we let
$$
\tilde{s}=\int_{-\infty}^s\4{w}(\tau)^{\frac{1}{2}}\,d\tau\quad\mbox{ and }\quad
\psi (\tilde{s})=\4{w}(s)^{\frac{1}{2}}\quad\forall s>-\infty.
$$
Then 
\begin{equation}\label{w-psi-diff-relation}
(\log\4{w}(s))_s=2\psi_{\tilde{s}},\qquad \psi_{\tilde{s}\tilde{s}}
=\frac{(\log\4{w}(s))_{ss}}{2\4{w}(s)^{\frac{1}{2}}}
\end{equation}
and $K_0$, $K_1$, are given by (P.25 of \cite{DS2}),
\begin{equation}\label{K0-K1-defn}
K_0=-\frac{\psi_{\tilde{s}\tilde{s}}}{\psi},\qquad 
K_1=\frac{1-\psi_{\tilde{s}}^2}{\psi^2}.
\end{equation}
By \eqref{w_s-limit} and Lemma \ref{R'<0},
\begin{equation}\label{Rs<0}
\beta(\log\4{w}(s))_{ss}=R_s=rR_r<0\quad\forall s>-\infty.
\end{equation}
By \eqref{w-psi-diff-relation}, \eqref{K0-K1-defn} and \eqref{Rs<0}, we get
$$
K_0>0\quad\forall r>0
$$
and
\begin{equation}\label{K0-eqn}
K_0=-\frac{R_s}{2\beta\4{w}}=-\frac{R_r}{2\beta rv^{1-m}(r)}.
\end{equation}
Let $Q(r)=v(r)^{1-m}R(r)(R(r)-\rho)/(n-1)$. Then by 
By \eqref{R-eqn}, \eqref{R-integral} and \eqref{K0-eqn},
\begin{align}\label{K0(0)}
K_0(0)=&\lim_{r\to 0}
\frac{\int_0^rz^{n-1}Q(z)e^{\frac{\beta}{n-1}\int_0^z\tau v(\tau)^{1-m}\,d\tau}\,dz}
{2\beta r^nv^{1-m}(r)e^{\frac{\beta}{n-1}\int_0^r\tau v(\tau)^{1-m}\,d\tau}}\notag\\
=&\frac{v^{m-1}(0)}{2\beta}\cdot\lim_{r\to 0}
\frac{\int_0^rz^{n-1}Q(z)e^{\frac{\beta}{n-1}\int_0^z\tau v(\tau)^{1-m}\,d\tau}\,dz}{r^n}
\notag\\
=&\frac{v^{m-1}(0)}{2\beta}\cdot\lim_{r\to 0}
\frac{r^{n-1}Q(r)e^{\frac{\beta}{n-1}\int_0^r\tau v(\tau)^{1-m}\,d\tau}}{nr^{n-1}}
\notag\\
=&\frac{R(0)(R(0)-\rho)}{2\beta n (n-1)}\notag\\
=&\frac{2\beta+\rho}{n(n-1)}.
\end{align}
Hence $K_0(r)>0$ for all $r\ge 0$. We will now show that $K_1$ is strictly
positive on $\R^n$. We first claim that $\psi_{\tilde{s}}(0)=1$. This result
is stated without proof in \cite{AK} and \cite{DS2}. For the sake of
completeness we will give a short simple proof here. By direct computation,
\begin{align}
&\psi_{\tilde{s}}=1+\frac{1-m}{2}\cdot\frac{rv_r(r)}{v(r)}\label{psi-eqn}\\
\Rightarrow\quad&\psi_{\tilde{s}}(0)=1\notag
\end{align}
and the claim follows. We next observe that by \eqref{w-psi-diff-relation}
and \eqref{Rs<0}, $\psi_{\tilde{s}\tilde{s}}<0$. Hence 
\begin{equation}\label{psi-1st-derviative-upper-bd}
\psi_{\tilde{s}}(\tilde{s})<\psi_{\tilde{s}}(0)=1\quad\forall \tilde{s}>0.
\end{equation}
By \eqref{psi-eqn} and (v) of Lemma \ref{soliton-bd-lem},
\begin{equation}\label{psi-1st-derviative-lower-bd}
\psi_{\tilde{s}}(\tilde{s})>0\quad\forall \tilde{s}\ge 0.
\end{equation}
By \eqref{K0-K1-defn}, \eqref{psi-1st-derviative-upper-bd} and 
\eqref{psi-1st-derviative-lower-bd}, 
$$
K_1(r)>0\quad\forall r>0.
$$
By \eqref{K0-K1-defn} and the l'Hospital rule,
\begin{equation}\label{K1(0)}
K_1(0)=\lim_{\tilde{s}\to 0}\frac{1-\psi_{\tilde{s}}^2}{\psi^2}
=-\lim_{\tilde{s}\to 0}\frac{\psi_{\tilde{s}\tilde{s}}}{\psi}
=K_0(0).
\end{equation}
By \eqref{K0(0)} and \eqref{K1(0)}, we get \eqref{K(0)} and $K_1>0$ for all
$r\ge 0$. We will now prove \eqref{K0-limit}. By (iii) of Theorem
\ref{u-R-limit-thm} and the result of \cite{H} there exists a constant 
$C>0$ such that
\begin{equation}\label{r2v(1-m)-lower-bd}
r^2v^{1-m}\ge\left\{\begin{aligned}
&C\qquad\quad\,\forall r\ge 2\quad\mbox{ if $g$ is a Yamabe shrinking 
soliton with $\beta>\frac{\rho}{n-2}$}\\
&C\log r\quad\forall r\ge 2
\quad\mbox{ if $g$ is a Yamabe steady soliton with $\alpha>0$}\\
&Cr^{|\rho|/\beta}\qquad\,\forall r\ge 2\quad
\mbox{ if $g$ is a Yamabe expanding soliton with $\alpha>0$}.
\end{aligned}\right.
\end{equation}
Hence we always have 
\begin{equation}\label{r^nv-limit}
r^nv^{1-m}(r)\to\infty\quad\mbox{ as }r\to\infty. 
\end{equation}
We now divide the proof of \eqref{K0-limit} into two cases.

\noindent{\bf Case 1}: 
$\int_0^{\infty}z^{n-1}Q(z)e^{\frac{\beta}{n-1}\int_0^z\tau v(\tau)^{1-m}\,d\tau}\,dz
<\infty$.

\noindent Then by \eqref{r^nv-limit},
\begin{equation}
\lim_{r\to\infty}K_0(r)=\lim_{r\to\infty}
\frac{\int_0^rz^{n-1}Q(z)e^{\frac{\beta}{n-1}\int_0^z\tau v(\tau)^{1-m}\,d\tau}\,dz}
{2\beta r^nv^{1-m}(r)e^{\frac{\beta}{n-1}\int_0^r\tau v(\tau)^{1-m}\,d\tau}}=0.
\end{equation}

\noindent{\bf Case 2}: 
$\int_0^{\infty}z^{n-1}Q(z)e^{\frac{\beta}{n-1}\int_0^z\tau v(\tau)^{1-m}\,d\tau}\,dz
=\infty$.

\noindent Then by \eqref{r^nv-limit} and the l'Hospital rule,
\begin{equation}\label{K0-infty-eqn}
\lim_{r\to\infty}K_0(r)=\lim_{r\to\infty}
\frac{\int_0^rz^{n-1}Q(z)e^{\frac{\beta}{n-1}\int_0^z\tau v(\tau)^{1-m}\,d\tau}\,dz}
{2\beta r^nv(r)^{1-m}e^{\frac{\beta}{n-1}\int_0^r\tau v(\tau)^{1-m}\,d\tau}}
=\frac{1}{2\beta}\lim_{r\to\infty}
\frac{r^{n-1}Q(r)}{E}
\end{equation}
where
\begin{equation}\label{E-eqn}
E=(r^2v(r)^{1-m})_rr^{n-2}+r^2v(r)^{1-m}\left[(n-2)r^{n-3}+r^{n-2}\cdot
\frac{\beta}{n-1}rv(r)^{1-m}\right].
\end{equation}
Since by (v) of Lemma \ref{soliton-bd-lem} and Lemma 3.1 of \cite{H},
$(r^2v^{1-m}(r))_r>0$ for all $r>0$, by \eqref{E-eqn},
\begin{equation}\label{E-lower-bd}
E\ge r^{n-1}v(r)^{1-m}\left((n-2)+\frac{\beta}{n-1}r^2v(r)^{1-m}\right).
\end{equation}
By \eqref{K0-infty-eqn} and \eqref{E-lower-bd},
\begin{align}\label{K0-infty-upper-bd}
0\le\lim_{r\to\infty}K_0(r)\le&\frac{1}{2\beta}\lim_{r\to\infty}
\frac{Q(r)}{v^{1-m}(r)(n-2+(\beta/(n-1))r^2v(r)^{1-m})}\notag\\
=&\frac{1}{2\beta(n-1)}\lim_{r\to\infty}
\frac{R(r)(R(r)-\rho)}{n-2+(\beta/(n-1))r^2v(r)^{1-m}}\notag\\
\le&\frac{1}{2\beta^2}\lim_{r\to\infty}
\frac{R(r)(R(r)-\rho)}{r^2v(r)^{1-m}}.
\end{align}
By Theorem \ref{u-R-limit-thm}, (iv) of Lemma \ref{soliton-bd-lem} and 
\eqref{r2v(1-m)-lower-bd}, the right hand side of \eqref{K0-infty-upper-bd} 
tends to zero as $r\to\infty$ and \eqref{K0-limit} follows. 
We will now prove \eqref{K1-limit}. We first 
observe that by \eqref{K0-K1-defn},
\begin{equation}\label{K1-upper-bd}
0\le K_1(r)\le\frac{1}{r^2v(r)^{1-m}}\quad\forall r>0.
\end{equation}
If $g$ is a Yamabe steady or expanding soliton with $\alpha>0$, then
by \eqref{r2v(1-m)-lower-bd} the right hand side of \eqref{K1-upper-bd}
tends to zero as $r\to\infty$ and hence
\begin{equation*}
\lim_{r\to\infty}K_1(r)=0.
\end{equation*}
Suppose now $g$ is a Yamabe shrinking soliton with $\beta>\frac{\rho}{n-2}$. 
Then by \eqref{w_s-limit}, \eqref{w-psi-diff-relation} and 
Theorem \ref{u-R-limit-thm},
\begin{equation}\label{psi-limit}
\psi_{\tilde{s}}(\tilde{s})=\frac{R-\rho}{2\beta}\to 0\quad\mbox{ as }
\,s=\log r\to\infty.
\end{equation} 
By \eqref{K0-K1-defn}, \eqref{psi-limit} and Theorem \ref{u-R-limit-thm},
\begin{equation*}
\lim_{r\to\infty}K_1(r)=\frac{\rho}{(n-1)(n-2)}
\end{equation*}
and the theorem follows.

\hfill$\square$

\vspace{6pt}

\begin{thebibliography}{99}

\bibitem[AK]{AK} S.~Angenent and D.~Knopf, {\em An example of neckpinching 
for Ricci flow on $S^{n+1}$}, Math. Res. Lett. 11 (2004),no. 4, 493--518.

\bibitem[C]{C} B.~Chow, {\em The Yamabe flow on locally conformally flat
manifolds with positive Ricci curvature}, Comm. Pure Appl. Math. 45 (1992),
1003--1014.

\bibitem[CMM]{CMM} G.~Catino, C.~Mantegazza, and L.~Mazzieri, {\em On the
global structure of conformal gradient solitons with nonnegative Ricci 
tensor}, http://arxiv.org/abs/1109.0243. 

\bibitem[DS1]{DS1} P.~Daskalopoulos and N.~Sesum, {\em On the extinction
profile of solutions to fast diffusion}, J. Reine Angew Math. 622 (2008),
95--119.

\bibitem[DS2]{DS2} P.~Daskalopoulos and N.~Sesum, {\em The classification of
locally conformally flat Yamabe solitons}, http://arxiv.org/abs/1104.2242.

\bibitem[GP]{GP} B.H.~Gilding and L.A.~Peletier, {\em On a class of similarity
solutions of the porous media equation}, J. Math. Analy. Appl. 55
(1976), 351--364.

\bibitem[H]{H} S.Y.~Hsu, {\em Singular limit and exact decay rate 
of a nonlinear elliptic equation}, Nonlinear Analysis TMA 75 (2012), no. 7,
3443--3455.

\bibitem[PS]{PS} M.~Del Pino and M.~S\'aez, {\em On the extinction profile 
for solutions of $u_t=\Delta u^{(N-2)/(N+2)}$}, Indiana Univ. Math. J. 50
(2001), no. 1, 611--628.

\bibitem[SY]{SY} R.~Schoen and S.T.~Yau, {\em Lectures on differential 
geometry}, Cambridge, MA, USA, International Press, 1994.

\bibitem[V1]{V1} J.L.~Vazquez, {\em Smoothing and decay estimates for nonlinear
diffusion equations}, Oxford Lecture Series in Mathematics and its Applications
33, Oxford University Press, Oxford, 2006.

\bibitem[V2]{V2} J.L.~Vazquez, {\em The porous medium equation-Mathematical 
Theory}, Oxford Mathematical Monographs, Oxford University Press, 2007.

\end{thebibliography}
\end{document}